\documentclass{amsart}
\usepackage{color}
\usepackage{amsmath}
\usepackage{amssymb}
\usepackage{amsthm}
\usepackage[msc-links,abbrev]{amsrefs}
\usepackage{hyperref}
\usepackage[noabbrev,capitalize]{cleveref}
\usepackage[english]{babel}
\usepackage{IEEEtrantools}
\usepackage{enumitem}

\newcommand\restr[2]{{
  \left.\kern-\nulldelimiterspace 
  #1 
  \vphantom{\big|} 
  \right|_{#2} 
  }}

\newcommand{\Ric}{\text{Ric}}

\def\sideremark#1{\ifvmode\leavevmode\fi\vadjust{\vbox to0pt{\vss
 \hbox to 0pt{\hskip\hsize\hskip1em
 \vbox{\hsize3cm\tiny\raggedright\pretolerance10000
 \noindent #1\hfill}\hss}\vbox to8pt{\vfil}\vss}}}

\newtheorem{theorem}{Theorem}[section]
\newtheorem{proposition}[theorem]{Proposition}
\newtheorem{lemma}[theorem]{Lemma}
\newtheorem{corollary}[theorem]{Corollary}

\theoremstyle{definition}
\newtheorem{definition}[theorem]{Definition}

\theoremstyle{remark}

\numberwithin{equation}{section}

\title{Weighted renormalized volume coefficients}
\author{Ayush Khaitan}
\address{417 McAllister Building \\ Penn State University \\ University Park, PA 16802 \\ USA}
\email{auk480@psu.edu}

\begin{document}
\maketitle
\begin{abstract}
We define weighted renormalized volume coefficients and prove that they are variational. We also prove that they can be written as polynomials of weighted extended obstruction tensors, the weighted Schouten tensor, and the weighted Schouten scalar. \end{abstract}

\section{Introduction}
The $\sigma_k(g^{-1}P)$-Yamabe problem, introduced by Viaclovsky in~\cite{Viaclovsky2000}, asks if there exists a metric with constant $\sigma_k$-curvature. Note that if $k\neq n/2$ and either $k\leq 2$ or $g$ is locally conformally flat, then $\sigma_k(g^{-1}P)=c$ is the Euler-Lagrange equation of the functional $\int_M \sigma_k(g^{-1}P)\, dv_g$ under conformal variation subject to the constraint that $\mathrm{Vol}_g(M)=1$. However, if $k\geq 3$ and $g$ is not locally conformally flat, then Branson and Gover proved in~\cite{BransonGover2008} that $\sigma_k(g^{-1}P)=c$ is not the Euler-Lagrange equation of any functional. 

Renormalized volume coefficients, denoted as $v_k$, first arose in the context of AdS/CFT correspondence~\cite{graham2000}. For a Poincar\'e-Einstein metric of the form $g_+=\frac{dr^2+g_r}{r^2}$, they are defined as $$\left(\frac{\operatorname{det} g_{\rho}}{\operatorname{det} g_{0}}\right)^{1 / 2} \sim 1+\sum_{k=1}^{\infty} v_{k} \rho^{k},$$ where $\rho=-\frac{1}{2}r^2$. In~\cite{ChangFang2008}, Chang and Fang show that $v_k$ is variational for $k\neq n/2$, and that $v_k=\sigma_k$ in the cases where $\sigma_k$ is variational; $v_{n/2}$ is shown to be variational in~\cites{BrendleViaclovsky2004 ,CaseLinYuan2016}. In addition to this, the conformal transformation law of $v_k$ is second order in the conformal factor~\cite{Graham2009} and Einstein metrics are ``stable" critical points of the total functional~\cite{ChangFangGraham2012}.

A \emph{smooth metric measure space} is a five-tuple $(M^n,g,f,m,\mu)$, where $(M^n,g)$ is a Riemannian manifold, $f$ is a smooth function defined on $M$, $m\in\mathbb{R}$ is a dimensional parameter, and $\mu\in\mathbb{R}_+$ is an auxiliary curvature parameter~\cites{Case2014s,MR3415769}. If $m\in\mathbb{N}$, a smooth metric measure space may be thought of as the warped product $(M^n\times F^m(\mu),g\oplus f^2 h)$ where $(F^m(\mu),h)$ is the $m$-dimensional simply connected spaceform of constant curvature $\mu$~\cite{Case2014s}. 

The \emph{space of metric measure structures}~\cite{Case2014s} on $(M^n,m,\mu)$ is $$\mathfrak{M}(M,m,\mu):=\mathrm{Met}(M)\times C^\infty(M; \mathbb{R}_+),$$ where $\mathrm{Met}(M)$ is the space of Riemannian metrics on $M$, and $C^\infty(M; \mathbb{R}_+)$ is the space of positive smooth functions on $M$.

Smooth metric measure spaces arise in many ways, including as (possibly collapsed) limits of sequences of Riemannian manifolds~\cite{CheegerColding2000a}, as smooth manifolds satisfying curvature-dimension inequalities~\cites{BakryEmery1985,Wei_Wylie}, as the geometric framework~\cite{CaseChang2013} for studying curved analogues of the Caffarelli--Silvestre extension for defining the fractional Laplacian~\cite{CaffarelliSilvestre2007}, and, in the limiting case $m=\infty$, as a geometric framework for the realization of the Ricci flow as a gradient flow~\cite{Perelman1}.

In \cite{Case2014s}, Case proved that a weighted version of $\sigma_k(g^{-1}P)$ is variational for $k\in\{1,2\}$ or under the weighted locally conformally flat assumption. 
It is natural to ask if one can similarly define a weighted version of $v_k$ that is variational. Case showed that this is true for $v_3$~\cite{Case2016v} in the case $m=\infty$.  Using the weighted ambient space and the weighted Poincar\'e space developed by Case and the author in~\cite{CaseKhaitan2022}, we show that this is true in general.

Let $(M^n,g,f,m,\mu)$ be a smooth metric measure space. When $m>0$, we set $\phi:=-m\ln f$, so that
\[ \mathrm{dvol}_\phi:=f^m \mathrm{dvol}_g = e^{-\phi}\,\mathrm{dvol}_g . \] The \emph{Bakry--\'Emery Ricci curvature} $Ric_\phi^m$ and the \emph{weighted scalar curvature} $R_\phi^m$ of $(M^n,g,f,m,\mu)$ are defined as
\begin{align*}
\Ric_\phi^m &= \Ric + \nabla^2\phi - \frac{1}{m} d\phi \otimes d\phi,\\
R_\phi^m &= R + 2\Delta\phi - \frac{m+1}{m}|\nabla\phi|^2 + m(m-1)\mu e^{2\phi/m} .
\end{align*}
Roughly speaking, a \emph{weighted ambient space} is a metric measure structure $(\widetilde{g},\widetilde{f})$ on $\mathbb{R}_+\times M^n\times (-\epsilon,\epsilon)$ such that 
\begin{enumerate}[label=(\roman*)]
\item if $n+m$ is not an even integer, then $\widetilde{\text{Ric}}_\phi^m,\widetilde{F_\phi^m}=O(\rho^\infty)$, where $\rho$ is the coordinate on $(-\epsilon,\epsilon)$ and $$\widetilde{F_\phi^m}:=\widetilde{f}\widetilde{\Delta}\widetilde{f}+(m-1)(|\widetilde{\nabla} \widetilde{f}|^2-\mu);$$ 

\item if $n+m$ is an even integer, then \begin{align*}\widetilde{\text{Ric}}_\phi^m,\widetilde{F_\phi^m}&=O(\rho^{\frac{n+m}{2}-1}),\\ g^{ij}(\widetilde{\text{Ric}}_\phi^m)_{ij}-{m}{f}^{-2}  \widetilde{F_\phi^m}&=O(\rho^{\frac{n+m}{2}}),
\end{align*} where $i,j$ denote the coordinates on $M$.  
\end{enumerate}
Consider a weighted ambient metric measure structure of the form (cf.\ \cite{CaseKhaitan2022})
\begin{equation}
\label{straight-normal-metric}
    \begin{aligned}
    \widetilde{g}&=2\rho dt^2+2td\rho dt+t^2 g_\rho,\\
    \widetilde{f}&=tf_\rho.
    \end{aligned}
\end{equation}
Metric measure structures of the form \cref{straight-normal-metric} are called \emph{straight} and \emph{normal}. Case and the author prove in~\cite{CaseKhaitan2022} that for $n+m\notin 2\mathbb{N}$, a straight and normal metric measure structure $(\widetilde{g},\widetilde{f})$ is uniquely determined modulo $O(\rho^\infty)$. For $n+m\in 2\mathbb{N}$, it is uniquely determined modulo $O(\rho^{\frac{n+m}{2}})$, while $\frac{1}{2}g^{kl}(g_\rho)_{kl}+\frac{m}{f}f_\rho$ is uniquely determined modulo $O(\rho^{\frac{n+m}{2}+1})$.

For a weighted ambient metric measure of the form \cref{straight-normal-metric}, we define \emph{weighted renormalized volume coefficients} $ v_{k,\phi}^m $ as the following
\begin{equation}
\label{volume-expanding}
\left(\frac{f_\rho}{f}\right)^m\left(\frac{\operatorname{det} g_\rho}{\operatorname{det} g}\right)^{\frac{1}{2}}=1+\sum\limits_{k=1}^\infty v_{k,\phi}^m \rho^k.
\end{equation}
The terms $ v_{k,\phi}^m $ are uniquely determined for all $k\geq 1$
if $n+m\notin 2\mathbb{N}$, and for $1\leq k\leq \frac{n+m}{2}$ if $n+m\in 2\mathbb{N}$. Note that we use a weighted ambient metric measure here instead of a weighted Poincar\'e metric measure. This is because we need a weighted ambient metric measure to define weighted extended obstruction tensors, which we require in \cref{first-theorem-wrvc,linear-expansion-theorem,conformal-change-volume-equation}.

\begin{definition}
\label{definition-weighted-extended-obstruction-tensors}
Let $k\geq 1$, $k<(n+m)/2-1$ for $n+m\in 2\mathbb{N}$. The $k^\text{th}$ \emph{weighted extended obstruction tensor} is \begin{equation*}(\Omega_\phi^m)^{(k)}_{ij}= \restr{\widetilde{R}_{\infty ij\infty,\underbrace{\scriptstyle \infty\dots\infty}_{k-1}}}{\rho=0}. \end{equation*}
\end{definition}

For a smooth metric measure space $(M^n,g,f,m,\mu)$, the \emph{weighted Schouten tensor $P_\phi^m$}, the \emph{weighted Schouten scalar $J_\phi^m$}, and $Y_\phi^m$~\cite{Case2014s} are defined as
\begin{align*}
    P_\phi^m & := \frac{1}{n+m-2}(\Ric_\phi^m - J_\phi^m g), \\
    J_\phi^m & := \frac{1}{2(n+m-1)}R_\phi^m ,\\
    Y_\phi^m& :=J_\phi^m-\mathrm{tr}_g P_\phi^m.
\end{align*}
The first two weighted renormalized volume coefficients are~(cf.\ \cite{Case2014s})
\begin{align*}
v_{1,\phi}^m&=J_\phi^m,\\
v_{2,\phi}^m&=\frac{1}{2}\big[(J_\phi^m)^2-|P_\phi^m|^2-\frac{1}{m}(Y_\phi^m)^2\big].
\end{align*}
In this article, we generalize the results of \cites{ChangFang2008,ChangFangGraham2012,Graham2009} to smooth metric measure spaces using the weighted ambient metric measure and weighted Poincar\'e metric measure constructed in~\cite{CaseKhaitan2022}.

Let $k\geq 1$, $k\leq (n+m)/2$ for $n+m\in 2\mathbb{N}$. We define 
$$ \mathcal{F}_{k,\phi}^m :=\int_M  v_{k,\phi}^m \, \mathrm{dvol}_\phi.$$

Consider the conformal transformation $(\widehat{g},\widehat{f})=(e^{2\omega}g,e^\omega f)$. We show that, given a straight and normal weighted ambient metric measure of the form of \cref{straight-normal-metric}, the conformal transformation formulas for $\restr{\partial_\rho^k}{\rho=0}g_\rho$, $\restr{\partial_\rho^k}{\rho=0}f_\rho$ and $v_{k,\phi}^m$ involve at most the second derivatives of $\omega$. Note that for a straight and normal weighted ambient metric, $\rho$ depends on the choice of metric measure structure $(g,f)$. Hence $\restr{\partial_\rho^k}{\rho=0}g_\rho$, $\restr{\partial_\rho^k}{\rho=0}f_\rho$ and $v_{k,\phi}^m$ are well-defined.
\begin{theorem}
\label{first-theorem-wrvc}
Let $k\geq 1$, and $k< (n+m)/2$ for $n+m\in 2\mathbb{N}$.
Under conformal change $\widehat{g}=e^{2 \omega} g$, the conformal transformation laws of $\partial_{\rho}^{k} g$ and  $\partial_{\rho}^{k} f$ involve at most the second derivatives of $\omega .$ The same is true for $\frac{1}{2} g^{i j} \partial^{k} g_{i j}+\frac{m}{f} \partial^{k}f$ and $v^m_{k,\phi}$ for $k\geq 1$, $k \leq (n+m)/ 2$ for $n+m\in 2\mathbb{N}$.
\end{theorem}
We prove \cref{first-theorem-wrvc} by showing that the coefficients of $(\widetilde{g},\widetilde{f})$ and $ v_{k,\phi}^m $ can be written as homogeneous polynomials in $P_\phi^m$, $Y_\phi^m$, and weighted extended obstruction tensors $(\Omega_\phi^m)^{(k)}$.

\begin{theorem}
\label{linear-expansion-theorem}
Let $k\geq 1$, $k<(n+m)/2$ for $n+m\in 2\mathbb{N}$.
There exist linear combinations $\mathcal{Q}$ and $\mathcal{S}$ of partial contractions of $Y_\phi^m, P_\phi^m$ and $(\Omega_\phi^m)^{(k)}$ with respect to $(g_\rho,f_\rho)$ such that
\begin{align*}\partial^k_{\rho}g_{ij}&=\mathcal{Q}(P_\phi^m,(\Omega_\phi^m)^{(1)}_{ij},\dots,(\Omega_\phi^m)^{(k-1)}_{ij}),\\
\partial^k_{\rho}f&=\mathcal{S}(Y_\phi^m,P_\phi^m,(\Omega_\phi^m)^{(1)}_{ij},\dots,(\Omega_\phi^m)^{(k-1)}_{ij}).\end{align*}
Similarly, for $k\geq 1$, $k\leq (n+m)/2$ for $n+m\in 2\mathbb{N}$, there exist linear combinations $\mathcal{T}$ and $\mathcal{V}_k$ of complete contractions of $Y_\phi^m, P_\phi^m$ and $(\Omega_\phi^m)^{(k)}$ with respect to $(g_\rho,f_\rho)$ such that
\begin{equation}
\begin{aligned}
\label{v-equation}
\frac{1}{2}g^{ij}\partial^{k} g_{ij}+\frac{m}{f}\partial^{k} f&=\mathcal{T}(Y_\phi^m,P_\phi^m,(\Omega_\phi^m)^{(1)}_{ij},\dots,(\Omega_\phi^m)^{(k-2)}_{ij}),\\
v_{k,\phi}^m &=\mathcal{V}_{k}(Y_\phi^m, P_\phi^m, (\Omega_\phi^m)^{(1)}, \ldots, (\Omega_\phi^m)^{(k-2)})
\end{aligned}
\end{equation}
\end{theorem}

For a smooth metric measure space $(M^n,g,f,m,\mu)$, and given $k\in \mathbb{N}$, the \emph{weighted $\sigma_k$-curvature} $\sigma_{k,\phi}^m(Y_\phi^m, g^{-1}P_\phi^m)$~\cite{Case2014s} is a functional defined on the space of metric measure structures. It is variational for $k\in\{1,2\}$, or when the metric measure space is locally conformally flat. For $m\in \mathbb{N}$, it is equal to $$\sigma_{k,\phi}^m(Y_\phi^m, g^{-1}P_\phi^m)=\sigma_k\Big(\underbrace{\frac{Y_\phi^m}{m},\dots,\frac{Y_\phi^m}{m}}_{m\text{ times}}, \lambda_1,\dots,\lambda_n\Big),$$ where $\sigma_k$ is a symmetric polynomial of degree $k$, and $\lambda_1,\dots,\lambda_n$ are the eigenvalues of $g^{-1}P_\phi^m$.

We relate the conformal transformation laws of $ v_{k,\phi}^m (g)$ and $\sigma^m_k(Y_\phi^m, g^{-1}P_\phi^m)$.  
\begin{theorem}
\label{conformal-change-volume-equation}
Let $k \geq 1$, $k \leq (n+m)/ 2$ for $n+m\in 2\mathbb{N}$. Then
\begin{equation*}
e^{2 k \omega} \widehat{v_{k,\phi}^m}=\widehat{\sigma_{k,\phi}^m}\left(\widehat{Y_{\phi}^m}, g^{-1} \widehat{P_{\phi}^m}\right)+\sum_{p=0}^{k-2} r_{k, p}(x, \nabla \omega,\widehat{Y_{\phi}^m}, \widehat{P_{\phi}^m}),
\end{equation*}
where $r_{k, p}(x, \nabla \omega, \widehat{Y_{\phi}^m}, \widehat{P_{\phi}^m})$ is a polynomial in $\omega_{i}, \widehat{Y_{\phi}^m}, \widehat{P_{\phi}^m}$ which is homogeneous of degree $p$ in $\widehat{Y_{\phi}^m}, \widehat{P_{\phi}^m}$, of degree at most $ 2k-2p-2$ in $\nabla_j \omega$, and with coefficients depending on $(g,f)$.
\end{theorem}Hence, the conformal transformation laws of $ v_{k,\phi}^m (g)$ and $\sigma^m_k(Y_\phi^m, g^{-1}P_\phi^m)$ are the same up to lower order terms.

Let $\mathcal{C}$ be the conformal class of the metric measure structure $(g,f)$. Then $\mathcal{C}_1\subset \mathcal{C}$ is the set of smooth metric measure structures with unit weighted volume. In order to study the critical points of the functional $ \mathcal{F}_{k,\phi}^m $ in $\mathcal{C}_1$, we first write down the variational formula for $ v_{k,\phi}^m $. 
\begin{theorem}
\label{infinitesimal-conformal-change-theorem}
Let $k \geq 1$, $k \leq (n+m)/2$ for $n+m\in 2\mathbb{N}$. The infinitesimal conformal variation of $v_{k,\phi}^m$ is
\begin{equation}
\label{infinitesimal-conformal-change}
\delta v_{k,\phi}^m =-2 \omega k v_{k,\phi}^m +\nabla^*_i[(L_{k,\phi}^m)^{ij}\nabla_j\omega],  
\end{equation} where $(L_{k,\phi}^m)^{ij}=-\frac{1}{k !} \restr{\partial_{\rho}^{k}}{{\rho=0}}\big[(v_\phi^m)\int_0^\rho g^{ij}(u) \,du\big]$ and $-\nabla^*$ is the adjoint of $\nabla$ with respect to the $L^2$-inner product induced by the weighted volume element~\cite{Case2014s}.
\end{theorem}
\cref{infinitesimal-conformal-change} allows us to study the critical points of $ \mathcal{F}_{k,\phi}^m $ in $\mathcal{C}_1$.
\begin{theorem}
\label{first-derivative-fphim}
Let $k\in \mathbb{N}$ be such that $k\neq \frac{n+m}{2}$. Then the critical points of $ \mathcal{F}_{k,\phi}^m $ in $\mathcal{C}_1$ are those for which $ v_{k,\phi}^m $ are constant.
\end{theorem}
\begin{definition}
\label{quasi-einstein-definition}
A smooth metric measure space $(M^n,g,f,m,\mu)$ is quasi-Einstein if 
\begin{equation}
\label{quasi-Einstein-conditions}
P_\phi^m=\lambda g,\quad J_\phi^m=(n+m)\lambda,
\end{equation}
for some $\lambda\in\mathbb{R}$.
\end{definition}
Quasi-Einstein spaces with $J_\phi^m\neq 0$ are local extrema of $\mathcal{F}_{k,\phi}^m$.
\begin{theorem}
\label{double-derivative-theorem}
For $1\leq k<d+m$, $k\neq \frac{n+m}{2}$ for $n+m\in 2\mathbb{N}$, let $(M^n,g,f,m,\mu)$ be a compact, connected quasi-Einstein space with a weighted volume of one such that $n+m\geq 3$ and $J_\phi^m\neq 0$. For $m=0$, we also assume that the space is not isometric to $S^n$ with the standard metric. Then the second variation of $\restr{ \mathcal{F}_{k,\phi}^m }{{\mathcal{C}_1}}$ is a definite quadratic form on $T_g\mathcal{C}_1$. More precisely:
\begin{enumerate}[label=(\roman*)]
\item If $1\leq k<\frac{n+m}{2}$, then
\begin{enumerate}
    \item if $J_\phi^m>0$, then $(\restr{ \mathcal{F}_{k,\phi}^m }{{\mathcal{C}_1}})''$ is positive definite,
    \item if $J_\phi^m<0$, then $(\restr{ \mathcal{F}_{k,\phi}^m }{{\mathcal{C}_1}})''$ is positive definite for $k$ odd and negative definite for $k$ even.
\end{enumerate}
\item If $\frac{n+m}{2}<k<n+m$, then
\begin{enumerate}
     \item if $J_\phi^m>0$, then $(\restr{ \mathcal{F}_{k,\phi}^m }{{\mathcal{C}_1}})''$ is negative definite,
    \item if $J_\phi^m<0$, then $(\restr{ \mathcal{F}_{k,\phi}^m }{{\mathcal{C}_1}})''$ is negative definite for $k$ odd and positive definite for $k$ even.
\end{enumerate}
\end{enumerate}
\end{theorem}

This article is organized as follows: in \cref{sec:weighted-extended-obstruction-tensors}, we discuss some properties of weighted extended obstruction tensors and prove \cref{first-theorem-wrvc,linear-expansion-theorem,conformal-change-volume-equation}. In \cref{sec:linearization}, we prove \cref{infinitesimal-conformal-change-theorem}. In \cref{sec:critical-points}, we prove \cref{first-derivative-fphim,double-derivative-theorem}.

\subsection*{Acknowledgements}
I would like to thank Professor Jeffrey S. Case for suggesting this problem, many helpful discussions and suggestions, and comments on multiple drafts of this paper.

\section{Weighted extended obstruction tensors}
\label{sec:weighted-extended-obstruction-tensors}
Given a weighted ambient space $\mathbb{R}_+\times M^n\times (-\epsilon,\epsilon)$, we use $0,\{i,j,k\}$ and $\infty$ as abstract indices for coordinates on $\mathbb{R}_+$, $M$, and $(-\epsilon,\epsilon)$ respectively.

First, we show that weighted extended obstruction tensors vanish for weighted locally conformally flat spaces.

\begin{lemma}
\label{lcf-lemma}
For a weighted locally conformally flat space, $(\Omega_\phi^m)^{(l)}=0$.
\end{lemma}
\begin{proof}
For a weighted locally conformally flat manifold, the weighted ambient space is flat~\cite{CaseKhaitan2022}*{Proposition 7.2}. Hence, $\Omega^{(l)}=0$.
\end{proof}

In general, the following transformation law holds for weighted ambient curvature tensors and their covariant derivatives.
\begin{proposition}
\label{curvature-transformation-law}
For $(M^n,m,\mu)$, let $(\widehat{g},\widehat{f})=(e^{2 \omega} g,e^\omega f)$ be a metric measure structure in the conformal class $[g,f]$. Also, let $IJKLM_{1} \cdots M_{r}$ be a list of indices, $s_{0}$ of which are $0, s_{M}$ of which are coordinates on $M$, and $s_{\infty}$ of which are $\infty .$ If $n+m\in 2\mathbb{N}$, assume that $s_{M}+2 s_{\infty} \leq n+m+1$. Then the conformal curvature tensors satisfy the conformal transformation law
\begin{equation*}
\restr{e^{2\left(s_{\infty}-1\right) \omega} \widehat{\widetilde{R}}_{I J K L, M_{1} \cdots M_{r}}}{\widehat{\rho}=0, \widehat{t}=1}=\restr{\widetilde{R}_{A B C D, F_{1} \cdots F_{r}}}{{\rho=0, t=1}} p_{I}^{A} \cdots p^{F_{r}}{ }_{M_{r}},
\end{equation*}
where $p_{I}^{A}$ is the matrix
\begin{equation}
\label{conformal-change-matrix}
p_{I}^{A}=\left(\begin{array}{ccc}
1 & \omega_{i} & -\frac{1}{2} \omega_{k} \omega^{k} \\
0 & \delta_{i}^{a} & -\omega^{a} \\
0 & 0 & 1
\end{array}\right).
\end{equation}
\end{proposition}
\begin{proof}
This is identical to the proof of \cite{FeffermanGraham2012}*{Proposition 6.5}.
\end{proof}

Now, we study the traces of weighted extended obstruction tensors.
\begin{theorem}
For $k\geq 1$, $k<\frac{n+m}{2}-1$ for $n+m\in 2\mathbb{N}$, we have
\begin{equation}
\label{trace-extended-obstruction-tensor}
g^{ij}(\Omega^m_\phi)^{(k)}_{ij}=-\big(\frac{m}{\widetilde{f}}(\widetilde{\nabla}^2 \widetilde{f})_{\infty\infty}\big)_{,\underbrace{\scriptstyle \infty\dots\infty}_{k-1}}.
\end{equation}
\end{theorem}
\begin{proof}
For a weighted ambient metric space, $(\widetilde{\text{Ric}_\phi^m})_{\infty\infty}=O(\rho^d)$ at $\rho=0$, where $d=\infty$ or $(n+m)/2-1$, for $n+m\notin 2\mathbb{N}$ or $n+m\in 2\mathbb{N}$, respectively. The conclusion follows from the identity \begin{equation*}g^{ij}\widetilde{R}_{\infty i j \infty,\underbrace{\scriptstyle\infty\dots \infty}_{k-1}}=-(\widetilde{\text{Ric}})_{\infty\infty,\underbrace{\scriptstyle\infty\dots \infty}_{k-1}}.\qedhere\end{equation*}
\end{proof}
\begin{definition}
Let $k \geq 1$, $k\leq (n+m)/2-1$ for $n+m\in 2\mathbb{N}$. We define the \emph{$k^{\text{th}}$ higher weighted Cotton tensor} $C_{i j l}^{(k)}$ as
\begin{equation}
\begin{aligned}
\label{cotton}
(C_\phi^m)_{i j l}^{(k)}=2 \widetilde{R}_{\infty(i j) l, \underbrace{\scriptstyle\infty\dots\infty}_{k-1}}+\widetilde{R}_{\infty i j \infty, \underbrace{\scriptstyle l \infty \ldots \infty}_{k-1}}+
\widetilde{R}_{\infty i j \infty, \underbrace{\scriptstyle\infty l \infty \ldots \infty}_{k-1}}\\
+\cdots+\widetilde{R}_{\infty i j \infty, \underbrace{\scriptstyle\infty\dots l}_{k-1}} .
\end{aligned}
\end{equation}
\end{definition}
We note that $(\widetilde{\text{Ric}_\phi^m})_{\infty l}=O(\rho^d)$ at $\rho=0$, where $d=\infty$ or $(n+m)/2-1$, for $n+m\notin 2\mathbb{N}$ or $n+m\in 2\mathbb{N}$, respectively. Also, $g^{ij}\widetilde{R}_{\infty i j l}=-\widetilde{\text{Ric}}_{\infty l}$. Hence, the trace of the $k^{\text{th}}$ higher weighted Cotton tensor can be written as
\begin{equation*}
\begin{aligned}
g^{ij}(C_\phi^m)^{(k)}_{ijl}=-\Big[2\big(\frac{m}{\widetilde{f}}(\widetilde{\nabla}^2 \widetilde{f}){}_{l\infty}\big){}_{,\underbrace{\scriptstyle\infty\dots\infty}_{k-1}}+\big(\frac{m}{\widetilde{f}}(\widetilde{\nabla}^2 \widetilde{f}){}_{\infty\infty}\big){}_{,\underbrace{\scriptstyle l\infty\dots\infty}_{k-1}}+\\
\quad\quad\dots+\big(\frac{m}{\widetilde{f}}(\widetilde{\nabla}^2 \widetilde{f}){}_{\infty\infty}\big){}_{{,\underbrace{\scriptstyle\infty\dots\infty l}_{k-1}}}\Big].
\end{aligned}
\end{equation*}
Using \cref{curvature-transformation-law}, we directly write down the conformal transformation law of weighted extended obstruction tensors.
\begin{equation}
\begin{aligned}
\label{obstruction-transformation-law}
e^{2 k \omega} (\widehat{\Omega_{\phi}^m})_{i j}^{(k)}&=(\Omega_\phi^m)_{i j}^{(k)}\\
&\quad\quad+\sum{}^{{}^{\prime}}{\restr{\widetilde{R}_{A B C D, F_{1} \cdots F_{k-1}}}{{\rho=0, t=1}} p_{\infty}^{A} p_{i}^{B}{p}_{j}^{C} p_{\infty}^{D} p_{\infty}^{F_{1}} \cdots p^{F_{k-1}}_\infty},
\end{aligned}
\end{equation}
where $p^{A} I$ is given by \cref{conformal-change-matrix} and $\sum^{\prime}$ denotes the sum over all indices except for $A B C D F_{1} \cdots F_{k-1}=\infty i j \infty \infty \cdots \infty$. Using \cref{cotton}, we can also write \cref{obstruction-transformation-law} as
\begin{equation*}
e^{2 k \omega} (\widehat{\Omega_{\phi}^m})_{i j}^{(k)}=(\Omega_\phi^m)_{i j}^{(k)}-\omega^{l} C_{i j l}^{(k)}+O\left(|\nabla \omega|^{2}\right).
\end{equation*}
We now show that the coefficients of the weighted ambient metric measure can be written as homogeneous polynomials in $P_\phi^m,Y_\phi^m$ and $(\Omega_\phi^m)^{(j)}$.
\begin{proof}[Proof of \cref{linear-expansion-theorem}]
Let $(\Lambda_\rho^{(k)})_{i j}=\widetilde{R}_{\infty i j \infty,\underbrace{\scriptstyle\infty \ldots \infty}_{k-1}}$. Hence, $\restr{\Lambda^{(k)}_\rho}{{\rho=0}}=(\Omega_\phi^m)^{(k)}$. We perform all our calculations at $\rho=0$. Let $k\geq 1$, and for $n+m\in 2\mathbb{N}$, let $k\leq (n+m)/2-1$ for $(\partial^k g_\rho,\partial^k f_\rho)$ and $k\leq (n+m)/2$ for $\frac{1}{2}g^{ij}\partial^{k} g_{ij}+\frac{m}{f}\partial^{k} f$. First, we prove that
\begin{equation}
\begin{aligned}
\label{linear-formulas}
\partial_{\rho}^{k} g_\rho&=\mathcal{Q}^{(k)}\left((g_\rho)^\prime, \Lambda^{(1)}_\rho, \ldots, \Lambda^{(k-1)}_\rho\right),\\
\partial_{\rho}^{k} f_\rho&=\mathcal{S}^{(k)}\left(f_\rho^\prime,(g_\rho)^{\prime}, \Lambda^{(1)}_\rho, \ldots, \Lambda^{(k-1)}_\rho\right),\\
\frac{1}{2}g^{ij}\partial^{k} g_{ij}+\frac{m}{f}\partial^{k} f&=\mathcal{T}\left(f_\rho^\prime,(g_\rho)^{\prime}, \Lambda^{(1)}_\rho, \ldots, \Lambda^{(k-2)}_\rho\right),
\end{aligned}
\end{equation} for some linear combinations of contractions $\mathcal{Q}^{(k)},\mathcal{S}^{(k)}$ and $\mathcal{T}^{(k)}$, whose coefficients are functions of $g,g^{-1}$ and $f$.

The $k=1$ case is trivially true. Let us assume that these formulas hold true for $(\partial^{k-1} g_{ij},\partial^{k-1} f)$, and consider $(\partial^{k} g_{ij},\partial^{k} f)$. For brevity, we refer to $\Lambda_\rho,f_\rho$ and $g_\rho$ as $\Lambda,f$ and $g$ respectively in this proof.

For a straight and normal weighted ambient space, we have
\begin{subequations}
\begin{align}
\begin{split}
\label{lambda-derivative}
\partial_{\rho} \Lambda_{i j}^{(k)}&=\Lambda_{i j}^{(k+1)}+g^{l m} g_{m(i}^{\prime} \Lambda_{j) l}^{(k)},
\end{split}\\
\begin{split}
\label{g-double-prime}
g''_{ij}&=2\Lambda_{ij}^{(1)}+\frac{1}{2}g^{kl}g'_{ik}g'_{jl},
\end{split}\\
\begin{split}
\label{g-inverse}
(g^{ij})'&=-g^{ik}g'_{kl}g^{lj},
\end{split}\\
\begin{split}
\label{f-double-prime}
f''&=-\frac{f}{m}g^{ij}\Lambda_{ij}^{(1)} \,\bmod\, O(\rho^n),
\end{split}
\end{align}
\end{subequations}
where $n=\infty$ or $(n+m)/2-1$, for $n+m\notin 2\mathbb{N}$ or $n+m\in 2\mathbb{N}$, respectively.

\cref{lambda-derivative} follows from direct computation, and the following Christoffel symbols for a straight and normal weighted ambient space~\cite{CaseKhaitan2022}*{7.13}: $$\widetilde{\Gamma}_{\infty\infty}^K=0,\quad \widetilde{\Gamma}_{\infty j}^k=\frac{1}{2}(g_\rho)^{kl}(g_\rho)'_{lj}.$$

\cref{g-double-prime} follows from the equation~\cite{FeffermanGraham2012}*{6.1} $$\widetilde{R}_{\infty i j \infty}=\frac{1}{2}\left(g''-\frac{1}{2}g^{kl}g'_{ik}g'_{jl}\right).$$ 

\cref{f-double-prime} follows from the fact that $\restr{(\widetilde{\text{Ric}_\phi^m})_{\infty\infty}}{\rho=0}=O(\rho^n)$, where $n=\infty$ or $(n+m)/2-1$, for $n+m\notin 2\mathbb{N}$ or $n+m\in 2\mathbb{N}$, respectively.
As $\widetilde{\Gamma}_{\infty\infty}^K=0$, it holds that $(\widetilde{\nabla}^2\widetilde{f})_{\infty\infty}=f''$. Also, $(\widetilde{\text{Ric}})_{\infty\infty}=\widetilde{g}^{IJ}\widetilde{R}_{I\infty J\infty}=-g^{ij}\widetilde{R}_{\infty i j \infty}=-g^{ij}\Lambda^{(1)}_{ij}$. We get the result on putting these together. 

From \cref{lambda-derivative,g-double-prime,f-double-prime,g-inverse}, we get
\begin{equation*}
\begin{aligned}
\partial^{k} g_{ij}&=2\Lambda_{ij}^{(k-1)}+\mathcal{G}^{(k)}(g',\Lambda^{(1)}_{ij},\dots,\Lambda^{(k-2)}_{ij}),\\
\partial^{k} f&=-\frac{f}{m}g^{ij}\Lambda_{ij}^{(k-1)}+\mathcal{F}^{(k)}(f',g',\Lambda^{(1)}_{ij},\dots,\Lambda^{(k-2)}_{ij}),\\
\frac{1}{2}g^{ij}\partial^{k} g_{ij}+\frac{m}{f}\partial^{k} f&=\mathcal{L}^{(k)}(f',g',\Lambda^{(1)}_{ij},\dots,\Lambda^{(k-2)}_{ij}).
\end{aligned}
\end{equation*}
Thus, by induction, we complete the proof of \cref{linear-formulas}. Note that for $n+m\in 2\mathbb{N}$, $\Lambda_{ij}^{(k)}$ is defined only for $1\leq k<(n+m)/2-1$. That is why in the case that $n+m\in 2\mathbb{N}$, $k\leq (n+m)/2-1$ for $(\partial^k (g_\rho),\partial^k f_\rho)$ and $k\leq (n+m)/2$ for $\frac{1}{2}g^{ij}\partial^{k} g_{ij}+\frac{m}{f}\partial^{k} f$.

Now, upon studying the Taylor expansion of \cref{volume-expanding}, we observe that each $ v_{k,\phi}^m $ can be written as a complete contraction of $\partial_\rho^l g_{ij}$ and $\partial_\rho^l f$ for $1\leq l\leq k-1$ and $\frac{1}{2} g^{i j} \partial^{k} g_{i j}+\frac{m}{f} \partial^{k} f$. Hence, for $k\geq 1$, $k\leq (n+m)/2$ for $n+m\in 2\mathbb{N}$, we have $$
v_{k,\phi}^m(g)=\mathcal{V}^{(k)}\left(Y_\phi^m, P_\phi^m, (\Lambda_\phi^m)^{(1)}, \ldots, (\Lambda_\phi^m)^{(k-2)}\right)
$$ for some linear combination of contractions $\mathcal{V}$.
Restricting to $\rho=0$ and recalling that $\restr{\Lambda^{(k)}_{ij}}{{\rho=0}}=(\Omega_\phi^m)_{ij}^{(k)}$ completes the proof of \cref{linear-expansion-theorem}.
\end{proof}
We now study the conformal transformation laws of $(g_\rho,f_\rho)$ and $ v_{k,\phi}^m $.
\begin{proof}[Proof of \cref{first-theorem-wrvc}]
Consider the following transformation laws (cf.\ \cite{Case2014s}), where $(\widehat{g},\widehat{f})=\left(e^{-\frac{2 \omega} {n+m-2}} g,e^{- \frac{\omega} {n+m-2}} f\right)$.
\begin{subequations}
\label{transformation-laws}
\begin{align}
\begin{split}
e^{-\frac{2 \omega}{n+m-2}} \widehat{J_{\phi}^{m}}&=J_{\phi}^{m}+\Delta_{\phi} \omega-\frac{1}{2}|\nabla \omega|^{2},    
\end{split}\\
\begin{split}
\widehat{P_{\phi}^{m}}&=P_{\phi}^{m}+\nabla^{2} \omega+\frac{1}{n+m-2} d \omega \otimes d \omega-\frac{1}{2(n+m-2)}|\nabla \omega|^{2} g,\end{split}\\
\begin{split}
e^{-\frac{2 \omega}{n+m-2}}\widehat{Y_{\phi}^m}&=Y_\phi^m-\langle \nabla\phi,\nabla\omega\rangle-\frac{m}{2(n+m-2)}|\nabla\omega|^2.    
\end{split}
\end{align}
\end{subequations}
Clearly, all transformation formulas involve at most the second derivatives of $\omega$. Also, we know from \cref{curvature-transformation-law} that the conformal transformation formulas of $(\Omega_\phi^m)^{(k)}$ involve only the first derivatives of $\omega$. Hence, from \cref{linear-expansion-theorem}, we conclude that the conformal transformation formulas of $\partial^k g_{\rho}$, $\partial^{k} f_{\rho}$ and $v_{k,\phi}^m$ involve at most the second derivatives of $\omega$. 
\end{proof}
We now write down the conformal transformation law of $ v_{k,\phi}^m $ more explicitly, counting the number of $\nabla\omega$'s appearing in each summand of the expression. 
\begin{proof}[Proof of \cref{conformal-change-volume-equation}]
Let $(\widehat{g},\widehat{f})=(s^2 g, sf)$ for some $s\in \mathbb{R}_+$. Then the straight and normal weighted ambient spaces of $(g,f)$ and $(\widehat{g},\widehat{f})$ are related by the diffeomorphism 
\begin{equation*}
\widehat{t}=ts^{-1}, \,\,\, \widehat{x}=x, \,\,\, \widehat{\rho}=\rho s^2.    
\end{equation*}
Hence, we can infer from \cref{volume-expanding} that $\widehat{v_{k,\phi}^m}= s^{-2k} v_{k,\phi}^m $.
Now we know from \cref{v-equation} that 
$$
v_{k,\phi}^m=\mathcal{V}_{k}\left(Y_\phi^m, P_\phi^m, (\Omega_\phi^m)^{(1)}, \ldots, (\Omega_\phi^m)^{(k-2)}\right).
$$
Consider a summand in the linear combination of contractions $\mathcal{V}_{k}$ which is of degree $(y,d_0,d_1,\dots, d_{k-2})$ in $(Y_\phi^m, P_\phi^m, (\Omega_\phi^m)^{(1)}, \ldots, (\Omega_\phi^m)^{(k-2)})$. We know from \cref{transformation-laws} that $\widehat{Y_{\phi}^m}=s^{-2}Y_\phi^m$, and from \cref{curvature-transformation-law} that $\widehat{(\Omega_\phi^m)^{(l)}}=e^{-2l}(\Omega_\phi^m)^{(l)}$. Moreover, there must be $d_0+\dots+d_{k-2}$ contractions in this summand, and $\widehat{g}^{-1}=s^{-2}g^{-1}$. 
Therefore, we get 
\begin{equation}
\label{k-value}
k=y+ \sum\limits_{l=0}^{k-2} (l+1)d_l.   
\end{equation}
We now consider
\begin{equation}
\label{equation-1}
v_{k,\phi}^m=\sum\limits_{p=0}^k \mathcal{V}_{k,p}\left(Y_\phi^m, P_\phi^m, (\Omega_\phi^m)^{(1)}, \ldots, (\Omega_\phi^m)^{(k-2)}\right),
\end{equation} where $\mathcal{V}_{k,p}$ refers to the monomials in $\mathcal{V}_k$ such that $y+d_0=p$. We infer from \cref{k-value} that $\mathcal{V}_{k,k-1}=0$. Also, for spaces that are locally conformally flat in the weighted sense, $\mathcal{V}_k=\mathcal{V}_{k,k}$ because  by \cref{lcf-lemma}, $(\Omega_\phi^m)^{(l)}=0$. Since $ v_{k,\phi}^m =\sigma_{k,\phi}^m(Y_\phi^m, P_\phi^m)$ for locally conformally flat manifolds (cf.\ \cite{Case2014s}), we infer that $\mathcal{V}_{k,k}=\sigma_{k,\phi}^m(Y_\phi^m, P_\phi^m)$. We therefore write \cref{equation-1} as \begin{equation*}
v_{k,\phi}^m=(\sigma_\phi^m)_k(Y_\phi^m, P_\phi^m)+\sum\limits_{p=0}^{k-2} \mathcal{V}_{k,p}\left(Y_\phi^m, P_\phi^m, (\Omega_\phi^m)^{(1)}, \ldots, (\Omega_\phi^m)^{(k-2)}\right).
\end{equation*}
Now for $(\widehat{g},\widehat{f})=(e^{2\omega}g, e^\omega f)$, where $\omega\in C^\infty(M)$, we get
\begin{equation*}
\widehat{(v_{\phi}^m)}_{k}(\widehat{g},\widehat{f})=\widehat{(\sigma_{\phi}^m)}_k(\widehat{Y_{\phi}^m}, \widehat{P_{\phi}^m})+\sum\limits_{p=0}^{k-2} \mathcal{V}_{k,p}\left(\widehat{Y_{\phi}^m},\widehat{P_{\phi}^m}, (\widehat{\Omega_{\phi}^m})^{(1)}, \ldots, (\widehat{\Omega_{\phi}^m})^{(k-2)}\right).
\end{equation*}
We now separate out all summands of the form $e^{l\omega}$ on the right. Note that $\widehat{\sigma_{k,\phi}^m}(\widehat{Y_{\phi}^m}, \widehat{P_{\phi}^m})=e^{-2k\omega}\widehat{(\sigma_{\phi}^m)}_k(e^{2\omega}\widehat{Y_{\phi}^m}, \widehat{P_{\phi}^m})$, where all contractions are with respect to $g$. Similarly, 
\begin{align*}
\widehat{\mathcal{V}_{k,p}}\big(\widehat{Y_{\phi}^m},\widehat{P_{\phi}^m}, &(\widehat{\Omega_{\phi}^m})^{(1)}, \ldots, (\widehat{\Omega_{\phi}^m})^{(k-2)}\big)\\
&=e^{-2k\omega}\widehat{\mathcal{V}_{k,p}}\big(e^{2\omega}\widehat{Y_{\phi}^m},\widehat{P_{\phi}^m}, e^{2\omega}(\widehat{\Omega_{\phi}^m})^{(1)}, \ldots, e^{2(k-2)\omega}(\widehat{\Omega_{\phi}^m})^{(k-2)}\big).    
\end{align*} If we set 
\begin{equation*}
r_{k,p}(x,\nabla \omega, \widehat{Y_{\phi}^m},\widehat{P_{\phi}^m})=\widehat{\mathcal{V}_{k,p}}\big(e^{2\omega}\widehat{Y_{\phi}^m},\widehat{P_{\phi}^m}, e^{2\omega}(\widehat{\Omega_{\phi}^m})^{(1)}, \ldots, e^{2(k-2)\omega}(\widehat{\Omega_{\phi}^m})^{(k-2)}\big),      
\end{equation*} where all the $\nabla\omega$'s come from $(\widehat{\Omega_{\phi}^m})^{(l)}$, then $r_{k,p}(x,\nabla \omega, \widehat{Y_{\phi}^m},\widehat{P_{\phi}^m})$ is a polynomial that is homogeneous in $\widehat{Y_{\phi}^m}, \widehat{P_{\phi}^m}$, with coefficients depending on $(g,f)$. 

We now bound its degree in $\nabla\omega$. Following the argument in  \cite{Graham2009}*{Proposition 2.6}, we conclude that each $e^{2l\omega}(\widehat{\Omega_{\phi}^m})^{(l)}$ is of degree at most $ 2l$ in $\nabla\omega$. It follows from \cref{k-value} that $r_{k,p}(x,\nabla \omega, \widehat{Y_{\phi}^m},\widehat{P_{\phi}^m})$ has degree in $\nabla\omega$ at most
\begin{equation*}
2\sum\limits_{l=1}^{k-2}ld_l=2\big(k-p-\sum\limits_{l=1}^{k-2}d_l\big).\end{equation*} As $y+d_0=p<k$, we can infer from \cref{k-value} that at least one of the $d_l$'s has to be non-zero. This implies that the degree of $r_{k,p}(x,\nabla \omega, \widehat{Y_{\phi}^m},\widehat{P_{\phi}^m})$ in $\nabla\omega$ is at most $ 2k-2p-2$.\qedhere
\end{proof}

\section{Linearization}
\label{sec:linearization}
We first define weighted Poincar\'e spaces.
\begin{definition}
A \emph{weighted Poincar\'e space}~\cite{CaseKhaitan2022}*{Definition 5.1} for $(M^n,[g,f],m,\mu)$ is a metric measure structure $(g_+,f_+)$ on $M\times [0,\epsilon)$ such that $(g_+,f_+)$ has $(M^n,[g,f],m,\mu)$ as conformal infinity, and
\begin{enumerate}[label=(\roman*)]
\item if $n+m\notin 2\mathbb{N}$, then 
\begin{equation*}
\begin{aligned}
\text{Ric}_{\phi}^m(g_+)+(n+m)g_+,
F_\phi^m(g_+)
-(n+m)f_+^2=O(r^\infty),\end{aligned} 
\end{equation*}

\item if $n+m\in 2\mathbb{N}$, then 
\begin{equation*}
\begin{aligned}
\text{Ric}_{\phi}^m(g_+)+(n+m)g_+,
F_\phi^m(g_+)
-(n+m)f_+^2&=O(r^{n+m-2}),\\
\mathrm{tr}_{g_+}\mathrm{Ric}_\phi^m(g_+)-\frac{m}{f_+^2}F_\phi^m(g_+)+(n+m)(n+m+1)&=O(r^{n+m}),
\end{aligned} 
\end{equation*} 
\end{enumerate}
where $F_\phi^m(g_+)=f\Delta_{g_+}f+(m-1)(|\nabla f_+|^2_{g_+}-\mu)$.
\end{definition}
A weighted Poincar\'e space in normal form can be written as $$(g_+,f_+)=\left(\frac{dr^2+g_r}{r^2},\frac{f_r}{r}\right)$$ with $(g_0,f_0)=(g,f)$.

We now write down the infinitesimal conformal variation formula of $ v_{k,\phi}^m $.
\begin{proof}[Proof of \cref{infinitesimal-conformal-change-theorem}]
Let $(M^n,g,f,m,\mu)$ be a smooth metric measure space, and let $$(g_+,f_+)=\left(\frac{dr^2+g_r}{r^2},\frac{f_r}{r}\right)$$ be a weighted Poincar\'e space in normal~\cite{CaseKhaitan2022} form with respect to $(g,f)$, such that $(g_0,f_0)=(g,f)$. If we choose another metric measure structure $(\widehat{g},\widehat{f})=(e^{2\omega}g,e^\omega f)$ in the same conformal class $[g,f]$, the weighted Poincar\'e space in normal form for it is $\left(\frac{dr^2+\widehat{g_r}}{r^2},\frac{\widehat{f_r}}{r}\right)$ such that $(\widehat{g_0},\widehat{f_0})=(\widehat{g},\widehat{f})$. There exists a diffeomorphism $\psi:X\to X$ such that $\restr{\psi}{M}=Id$ and $$\left(\psi^*\frac{dr^2+g_r}{r^2},\psi^*\frac{f_r}{r}\right)=\left(\frac{dr^2+\widehat{g_r}}{r^2},\frac{\widehat{f_r}}{r}\right).$$
Let us now consider a one-parameter family of conformal transformations $(\widehat{g}(t),\widehat{f}(t))=(e^{2\omega t}g,e^{\omega t} f)$. This generates a one-parameter family of diffeomorphisms $\psi^*_t$ such that $\restr{\psi_t}{M}=Id$ and $$\left(\psi_t^*\frac{dr^2+g_r}{r^2},\psi_t^*\frac{f_r}{r}\right)=\left(\frac{dr^2+\widehat{g(t)}_r}{r^2},\frac{\widehat{f(t)}_r}{r}\right),$$ where $(\widehat{g(t)}_0,\widehat{f(t)}_0)=(e^{2\omega t}g,e^{\omega t} f)$. This one-parameter family of diffeomorphisms also generates a vector field $X=X^0\partial_r + X^i\partial_i$ such that the flow along $X$ corresponds to the family of diffeomorphisms. 

From the proof of \cite{Graham2009}*{Theorem 3.1}, we know that $$X^0=-\omega r,\quad X^i=\omega_j\int_0^r {sg_s^{ij}\,ds}\, ,$$ and
\begin{equation*}
(\delta g_r)_{ij}=\omega(2-r\partial_r)(g_r)_{ij}+2\nabla_{(i} X_{j)}.
\end{equation*}
Similarly, from
\begin{equation*}
L_X\left(\frac{f_r}{r}\right)=\frac{\delta f_r}{r},
\end{equation*}
we get 
\begin{equation*}
\delta f_r=\omega(1-r \partial_r) f_r+\omega_j(f_r)_i \int_0^rsg_s^{ij}\,ds.
\end{equation*}
With the change of variables $r=-\frac{1}{2}\rho^2$, we get
$$(v_\phi^m)(\rho,t)=\left(\frac{f_{\rho}(t)}{f_0(t)}\right)^{m}\left(\frac{\operatorname{det} g_{\rho}(t)}{\operatorname{det} g_0(t)}\right)^{\frac{1}{2}}.$$  Let $Y^i=-(\partial_{j} \omega)\int_{0}^{\rho} g^{i j}(u) \,d u$. We have 
\begin{align*}
\frac{\delta v_\phi^m}{v_\phi^m}=\delta \log (v_\phi^m)&=m\left(\frac{\delta f_\rho(t)}{f_\rho(t)}-\omega\right)+\frac{1}{2}\left(g^{i j} \delta g_{i j}-2 n \omega\right)\\
&=m\left(\frac{Y^i\partial_i f_\rho}{f_\rho}-2\omega\rho\frac{\partial_\rho f_\rho}{f_\rho}\right)-\omega g^{i j} \rho \partial_{\rho} g_{i j}+ \nabla_{i}^{(\rho)} Y^i\\
&=m\frac{Y^i\partial_i f_\rho}{f_\rho}-2\omega\rho\frac{\partial_\rho v_\phi^m}{v_\phi^m}+\nabla_{i}^{(\rho)} Y^i.
\end{align*}
Therefore, we get 
\begin{equation}
\label{delta-v}
\delta v_\phi^m= m v_\phi^m\frac{Y^i\partial_i f_\rho}{f_\rho}-2\omega\rho \partial_\rho v_\phi^m+v_\phi^m\nabla_i^{(\rho)}Y^i.   
\end{equation}
We observe that 
\begin{equation*}
v_\phi^m \nabla_i^{(\rho)}Y^i =\nabla_i^{(0)}(v_\phi^m Y^i)+m v_\phi^m\frac{\partial_i(f_0/f_\rho)}{(f_0/f_\rho)} Y^i.    
\end{equation*}

Hence, \cref{delta-v} can be written as \begin{equation}
\label{actual-delta-v}
   \delta v_\phi^m=-2\omega\rho \partial_\rho v_\phi^m+ \nabla_i^{(0)}(v_\phi^m Y^i)+mv_\phi^mY^i \partial_i(\log f_0).
\end{equation}
Note that as $\phi=-m\log f$ and $\nabla^*=\nabla-\nabla\phi$~\cite{Case2014s}, \cref{actual-delta-v} can also be written as 
\begin{equation*}
 \delta v_\phi^m=-2\omega\rho \partial_\rho v_\phi^m+ {\nabla^*_i}^{(0)}(v_\phi^m Y^i).
\end{equation*}
Therefore, 
\begin{equation}
\label{variation of volumetric coefficient}
\delta v_{k,\phi}^m=-2 \omega k v_{k,\phi}^m +\frac{1}{k !} \partial_{\rho}^{k}[{\nabla^*_i}^{(0)}(v_\phi^m Y^i)].
\end{equation}
We may also write \cref{variation of volumetric coefficient} as 
\begin{equation*}
\delta v_{k,\phi}^m =-2 \omega k v_{k,\phi}^m +\nabla^*_i[(L_{k,\phi}^m)^{ij}\nabla_j\omega],  
\end{equation*} where $(L_{k,\phi}^m)^{ij}=-\frac{1}{k !} \restr{\partial_{\rho}^{k}}{{\rho=0}}\big[(v_\phi^m)\int_0^\rho g^{ij}(u) \,du\big]$.
\end{proof}
\begin{corollary}
As $\frac{1}{k !} \partial_{\rho}^{k}[{\nabla^*_i}^{(0)}((v_\phi^m) Y^i)]$ is a divergence term, we have 
\begin{equation}
\label{variational-integral}
\int_M \delta v_{k,\phi}^m \,\mathrm{dvol}_\phi=-2k\int_M \omega \,v_{k,\phi}^m \,\mathrm{dvol}_\phi. 
\end{equation}
\end{corollary}
\section{Critical points}
\label{sec:critical-points}
In this section, we study the critical points of $ \mathcal{F}_{k,\phi}^m $. First, we show that the critical points of $ \mathcal{F}_{k,\phi}^m $ in $\mathcal{C}_1$ satisfy the condition that $ v_{k,\phi}^m $ is constant.
\begin{proof}[Proof of \cref{first-derivative-fphim}]
\label{proof-of-first-derivative-fphim}
We have 
\begin{equation*}
\delta  \mathcal{F}_{k,\phi}^m =\int_M \delta v_{k,\phi}^m \,\mathrm{dvol}_\phi+(n+m)\int_M v_{k,\phi}^m \omega\, \mathrm{dvol}_\phi.
\end{equation*}
From \cref{variational-integral}, we get
\begin{equation}
\label{variation-fphim}
\delta \mathcal{F}_{k,\phi}^m =(n+m-2k)\int_M v_{k,\phi}^m  \omega\,\mathrm{dvol}_\phi. \end{equation}
Let $\mathrm{WtdVol}$ denote the weighted volume functional. Using the Lagrange multiplier method, we have for some $a>0$,
\begin{align*}
\delta\left[\mathcal{F}_{k,\phi}^m-a \mathrm{WtdVol}(M)\right](g,f)=0.\end{align*}
Hence, \begin{align*}
(n+m-2k)\int_M v_{k,\phi}^m  \omega\,\mathrm{dvol}_\phi-a (n+m)\int_M \omega \,\mathrm{dvol}_\phi=0.
\end{align*} Therefore, we get $ v_{k,\phi}^m =\frac{a(n+m)}{(n+m-2k)}.\qedhere$
\end{proof}
\begin{lemma}
For $k=(n+m)/2$, the functional $ \mathcal{F}_{k,\phi}^m $ is conformally invariant.
\end{lemma}
\begin{proof}
This follows directly from \cref{variation-fphim}.
\end{proof}
Given a metric measure structure $(g,f)$, consider the following functional on the conformal class $[g,f]$: $$\mathfrak{F}_{k,\phi}^m(e^{2u}g,e^{u}f)=\int_0^1\int_M u\, v_{k,\phi}^m{}^{(e^{2su}g,e^{su}f)} \, \,\mathrm{dvol}{}_{(e^{2su}g,e^{su}f)}ds.$$
\begin{corollary}
For $k=(n+m)/2$, the critical points of $\mathfrak{F}_{k,\phi}^m$ in $\mathcal{C}_1$ are those for which $v_{k,\phi}^m$ are constant.
\end{corollary}
\begin{proof}
If $(e^{2u}g,e^u f)$ is a critical point of $\mathfrak{F}_{k,\phi}^m$ in $\mathcal{C}_1$, then $$\delta\left[\mathfrak{F}_{k,\phi}^m-a \mathrm{WtdVol}(M)\right](e^{2u}g,e^uf)=0$$ for some $a>0$.
For $k=(n+m)/2$, we get
\begin{align*}
\delta \mathfrak{F}_{k,\phi}^m(e^{2u}g,e^{u}f)&=\restr{\frac{\partial}{\partial t}}{t=0}\mathfrak{F}_{k,\phi}^m(e^{2(u+t\omega)}g,e^{(u+t\omega)}f)\\
&=\int_0^1 \int_M \restr{\frac{\partial}{\partial t}}{t=0} (u+t\omega) v_{k,\phi}^m{}^{(e^{2s(u+t\omega)}g,e^{s(u+t\omega)}f)}\mathrm{dvol}{}_{(e^{2s(u+t\omega)}g,e^{s(u+t\omega)}f)}ds\\
&=\int_0^1\int_M \Big[\omega v_{k,\phi}^m{}^{(e^{2su}g,e^{su}f)}+su\left[\nabla{}_i(L^{ij}\nabla_j\omega)\right]^{(e^{2su}g,e^{su}f)}\Big]\mathrm{dvol}_{(e^{2su}g,e^{su}f)}ds\\
&=\int_0^1 \frac{\partial}{\partial s}\left[\int_M s\omega v_{k,\phi}^m{}^{(e^{2su}g,e^{su}f)}\mathrm{dvol}_{(e^{2su}g,e^{su}f)}\right]ds\\
&=\int_M \omega v_{k,\phi}^m{}^{(e^{2u}g,e^{u}f)}\mathrm{dvol}_{(e^{2u}g,e^{u}f)}.
\end{align*}
Hence, 
\begin{multline*}
\delta\left[\mathfrak{F}_{k,\phi}^m-a \mathrm{WtdVol}(M)\right](e^{2u}g,e^uf)=\\
\int_M [(v_{k,\phi}^m){}^{(e^{2u}g,e^{u}f)}-a(n+m)] \omega \,\mathrm{dvol}_{(e^{2u}g,e^{u}f)}.
\end{multline*}
Hence, if $(e^{2u}g,e^u f)$ is a critical point of $\mathcal{F}_{k,\phi}^m$ for $k=(n+m)/2$, then $(v_{k,\phi}^m){}^{(e^{2u}g,e^{u}f)}$ is constant.
\end{proof}

We now show that quasi-Einstein spaces with nonzero weighted scalar curvature are local extrema of $\mathcal{F}_{k,\phi}^m$ in $\mathcal{C}_1$. Note that the tangent space of $\mathcal{C}_1$ at $(g,f)$ is \begin{equation}
\label{tangent-space-c1-definition}
    \mathcal{T}_{(g,f)}\mathcal{C}_1=\{(2\omega g,\omega f):\int_M \omega \,\mathrm{dvol}_\phi=0\}.
\end{equation}
\begin{lemma}
Quasi-Einstein spaces are critical points of $\mathcal{F}_{k,\phi}^m$ in $\mathcal{C}_1$.
\end{lemma}
\begin{proof}
For a quasi-Einstein space of the form of \cref{quasi-einstein-definition}, there is~\cite{CaseKhaitan2022}*{Section 7.1} a unique straight and normal weighted ambient space corresponding to it, of the form \cref{straight-normal-metric}, such that 
\begin{equation}
\label{quasi-einstein-ambient}
g_\rho=(1+\lambda\rho)^2 g, \quad f_\rho=(1+\lambda\rho)f.
\end{equation}
Consequently, as $(v_\phi^m)(\rho)=(1+\lambda\rho)^{n+m}$, we deduce that
\begin{equation}
\begin{aligned}
\label{einstein_v}
v_{k,\phi}^m={n+m\choose k}\lambda^{k},
\end{aligned}
\end{equation}
where $${n+m\choose k}=\frac{(n+m)(n+m-1)\dotsm (n+m-k+1)}{k!}.$$
In particular, $v_{k,\phi}^m$ is constant. The conclusion follows from \cref{first-derivative-fphim}.
\end{proof}

Now we write down the formula for $(\left.\mathcal{F}_{k,\phi}^m\right|_{\mathcal{C}_{1}})^{\prime \prime}(\omega)$. 
\begin{lemma}
Let $n+m \geq 3, k \geq 1$ and $k \leq (n+m)/2$ if $n+m\in 2\mathbb{N}$. Let $(M^n, g,f,m,\mu)$ be a connected compact Riemannian metric measure space and suppose $(g,f)$ satisfies $v_{k,\phi}^m=c$ for some constant $c$. Let $\omega \in C^{\infty}(M)$ satisfy $\int_{M} \omega \,d (v_\phi^m)_{g}=0$. Then
\begin{equation}
\label{double-derivative}
(\left.\mathcal{F}_{k,\phi}^m\right|_{\mathcal{C}_{1}})^{\prime \prime}(\omega)=-(n+m-2k)\int_M[2k  v_{k,\phi}^m \omega^2 +(L_\phi^m)^{ij}\nabla_i \omega\nabla_j \omega]\, \mathrm{dvol}_\phi.    
\end{equation}
\end{lemma}
\begin{proof}
Let $\gamma_t$ be a curve in $\mathcal{C}_1$ such that $\gamma_0=(g,f)$ and $\gamma'_0=(2\omega g,\omega f)$. Then 
\begin{equation*}
(\left.\mathcal{F}_{k,\phi}^m\right|_{\mathcal{C}_{1}})^{\prime \prime}(\omega)=\left.\partial_{t}^{2}\right|_{t=0}\left((\mathcal{F}_\phi^m)_{k}-a \mathrm{WtdVol}(M)\right)\left(g_{t},f_t\right).    
\end{equation*}
Note that here $ v_{k,\phi}^m =\frac{a(n+m)}{(n+m-2k)}$. Using \cref{variation-fphim} and integration by parts, we get \cref{double-derivative}.
\end{proof}

Before we prove \cref{double-derivative-theorem}, we must prove a weighted version of the Lichnerowicz-Obata theorem~\cite{Case2014s}. Bakry and Qian proved that when the weighted Ricci curvature of a smooth metric measure space is bounded below, the first non-zero eigenvalue of the weighted Laplacian is also bounded below~\cite{BakryQian2000}*{Equation 3.1}. The following rigidity statement is well known to the experts, but we were not able to find a proof in the literature.
\begin{lemma}
\label{eigenvalue-lower-bound}
Let $(M^n,g,f,m,\mu)$ be a compact, connected smooth metric measure space such that $\mathrm{Ric}_\phi^m\geq 2(m+n-1)\lambda g$ for $\lambda>0$. Then the first non-zero eigenvalue $\lambda_1$ of $-\Delta_\phi$ satisfies $\lambda_1\geq 2(m+n)\lambda$. Moreover, if $\lambda_1=2(m+n)\lambda$, then $m=0$ and $M^n$ is isometric to $S^n(\sqrt{2\lambda})$.
\end{lemma}
\begin{proof}
We deduce from~\cite{Wei_Wylie}*{Equation 3.1} that $\lambda_1\geq 2(m+n)\lambda$. Moreover, if $\lambda_1=2(m+n)\lambda$ and $\Delta_\phi u=-2(m+n)\lambda u$, then~\cite{Wei_Wylie}*{Equation 2.7}
\begin{equation}
\label{hess-equation}
    \nabla^2 u=-2\lambda ug. 
\end{equation} 
Using \cref{hess-equation} and Obata's Theorem~\cite{Obata1962}, we deduce that $(M^n,g)$ is isometric to $S^n(\sqrt{2\lambda})$. Hence, $\mathrm{Ric}=2(n-1)\lambda g$. Moreover, as $\mathrm{Ric}-\frac{m}{f}\nabla^2f=\mathrm{Ric}_\phi^m$, we have
\begin{align*} \frac{m}{f}\nabla^2 f&\leq -2 m\lambda g.\end{align*}
If $m\neq 0$, as $f>0$, we get $\Delta f\leq-2\lambda nf<0$. This is a contradiction.
\end{proof}

\begin{proof}[Proof of \cref{double-derivative-theorem}]
From \cref{quasi-einstein-ambient}, we get  $$v_\phi^m(\rho)\int_0^\rho g^{ij}(u)\, du=(1+\lambda\rho)^{n+m-1}\rho g^{ij}.$$ Hence, for any $m\geq 0$,
\begin{equation}
\begin{aligned}
\label{einstein_L}
(L_{k,\phi}^m)^{ij}&=-{n+m-1\choose k-1}\lambda^{k-1}g^{ij}.
\end{aligned}
\end{equation}
Using \cref{einstein_L,einstein_v}, we write \cref{double-derivative} as
\begin{equation}
\label{negative-case}
(\left.\mathcal{F}_{k,\phi}^m\right|_{\mathcal{C}_{1}})^{\prime \prime}(\omega)=c_k\lambda^{k-1}\int_M \big(|\nabla\omega|_g^2-2(n+m)\lambda \omega^2\big)\, \mathrm{dvol}_\phi,
\end{equation} where $$c_k=(n+m-2k){n+m-1\choose k-1}.$$

If $J_\phi^m<0$, then $\lambda<0$ as well. Hence, we find that the integrand in \cref{negative-case}, $|\nabla\omega|_g^2-2J_\phi^m\omega^2$, is positive. Consequently, $(\left.\mathcal{F}_{k,\phi}^m\right|_{\mathcal{C}_{1}})^{\prime \prime}(\omega)$ has the same sign as $c_k\lambda^{k-1}$. 

If $J_\phi^m>0$, then $\lambda>0$ as well. \cref{eigenvalue-lower-bound} implies that $\lambda_1>2(n+m)\lambda$. Consequently, the integrand in \cref{negative-case} is positive. Therefore, $(\left.\mathcal{F}_{k,\phi}^m\right|_{\mathcal{C}_{1}})^{\prime \prime}(\omega)$ has the same sign as $c_k$.
\end{proof}

\bibliographystyle{plain}
\bibliography{bib.bib}
\end{document}